\documentclass[11pt]{amsart}
\usepackage{epsfig}

\usepackage{stmaryrd}
\usepackage{graphicx}
\usepackage{amscd}
\usepackage{amsmath}
\usepackage{amsfonts}
\usepackage{mathrsfs}
\usepackage{amssymb}
\usepackage{verbatim}
\usepackage{color}
\usepackage{hyperref}

\newtheorem{theorem}{Theorem}[section]
\theoremstyle{plain}

\newtheorem{corollary}[theorem]{Corollary}

\newtheorem{defi}[theorem]{Definition}
\newtheorem{example}[theorem]{Example}

\newtheorem{lemma}[theorem]{Lemma}

\newtheorem{prop}[theorem]{Proposition}

\numberwithin{equation}{section}

\def\what{\widehat}

\def\mfr{{\mathfrak m}}
\def\FrS{{\mathfrak S}}

\def\One{{1\!\!1}}

\def\wtil{\widetilde}
\def\half{\frac{1}{2}}

\newcommand{\lam}{\lambda}

\newcommand{\gam}{\gamma}
\newcommand{\om}{\omega}

\newcommand{\sig}{\sigma}

\newcommand{\R}{{\mathbb R}}
\newcommand{\Q}{{\mathbb Q}}
\newcommand{\Z}{{\mathbb Z}}
\newcommand{\C}{{\mathbb C}}

\def\N{{\mathbb N}}

\newcommand{\Nat}{{\mathbb N}}

\def\wt{\widetilde}
\def\A{{\mathcal A}}

\def\Ik{{\mathcal I}}

\def\Sf{{\sf S}}
\def\T{{\mathbb T}}

\def\bxi{{\mathbf \xi}}

\def\bz{{\mathbf z}}
\def\one{\vec{1}}
\def\be{\begin{equation}}
\def\ee{\end{equation}}
\newcommand{\Ek}{{\mathcal E}}

\newcommand{\eps}{{\varepsilon}}
\newcommand{\es}{\emptyset}

\def\und{\underline}

\def\Cc{{\mathscr M}}
\def\Mc{{\mathscr M}}

\def\a{a}

\def\ve1{\vec{1}}

\def\Ak{{\mathcal A}}

\def\bn{{\bf n}}
\def\bt{{\bf t}}

\begin{document}

\title{On substitution automorphisms with pure singular spectrum}

\author{Alexander I. Bufetov}
\address{Alexander I. Bufetov\\ 
Aix-Marseille Universit{\'e}, CNRS, Centrale Marseille, I2M, UMR 7373\\
39 rue F. Joliot Curie Marseille France }
\address{Steklov  Mathematical Institute of RAS, Moscow}
\address{Institute for Information Transmission Problems, Moscow}
\email{alexander.bufetov@univ-amu.fr, bufetov@mi.ras.ru}
\author{Boris Solomyak }
\address{Boris Solomyak\\ Department of Mathematics,
Bar-Ilan University, Ramat-Gan, Israel}
\email{bsolom3@gmail.com}

\begin{abstract}
A sufficient condition for a substitution automorphism to have pure singular spectrum is given in terms of the top Lyapunov exponent of the associated spectral cocycle. 
As a corollary, singularity of  the spectrum is established for an infinite family of self-similar interval exchange transformations.
 \end{abstract}

\date{\today}


\maketitle

\thispagestyle{empty}

\section{Introduction}
This paper is devoted to the spectral theory of  substitution systems.  Our main result is an explicit construction of a family of substitution automorphisms with purely singular spectrum.  In \cite{BuSo14} we gave a formula for spectral measures of ``cylindrical'' functions in the form of generalized matrix Riesz products. 
%We then obtained  H\"older continuity of spectral measures for generic $\R$-actions (suspension flows with a piecewise constant roof function) and log-H\"older estimates for self-similar flows, arising from non-Pisot substitutions.
In  \cite{BuSo19} we introduced a {\em spectral cocycle}, which is a generalization of the Riesz products mentioned above, not just for substitutions, but also for $S$-adic systems and generic translation flows on flat surfaces, and obtained formulas and estimates for the dimensions of spectral measures in terms of Lyapunov exponents. As a corollary, we obtained a sufficient condition for singularity of spectral measures for substitution $\R$-actions. Note, however, that there is in general no relation between the spectral type  of a transformation and that of a suspension flow over this transformation: for example, Kolmogorov 
constructed mixing suspension flows over circle rotations. Our previous results for $\R$-actions consequently do not give any information   about the spectrum of substitution $\Z$-actions.
This paper introduces a new idea: namely, using the spectral cocycle, we obtain the  singularity of the spectral type for substitution $\Z$-actions from  results on asymptotic equidistribution
under the action of ergodic toral endormorphisms, due to Host \cite{Host} and to Meiri \cite{Meiri}, see Section 3. Our main result, Theorem 2.4, states that if $\zeta$ is a primitive aperiodic substitution on two or more symbols such that the substitution matrix $\Sf_\zeta$ is irreducible and the logarithm of its Perron-Frobenius eigenvalue is strictly larger than twice the top Lyapunov exponent of the spectral cocycle, then the substitution $\Z$-action has pure singular spectrum. We apply our sufficient condition to an explicit family of examples, including those associated with self-similar interval exchange transformations of Salem type that are related to pseudo-Anosov diffeomorphisms, see Section 5.2.
At the same time, it is not clear whether our sufficient condition is necessary, and determining the spectral type of general substitution actions remains an intriguing open problem. 

Spectral properties of substitutions are still far from being understood completely, especially in the non-constant length case. Substitution systems are never strongly mixing \cite{DK}, hence their spectrum has a singular component. Although an absolutely continuous component seems to be rare (Rudin-Shapiro substitution, see \cite{Queff}, and its generalizations, due to Frank \cite{Frank}, and more recent ones of Chan and Grimm \cite{CG2}, are all of constant length), it is a non-trivial question to decide when the spectrum is purely singular. In the constant length there is a criterion of Queffelec \cite{Queff}, but few general results.
Abelian bijective substitutions is one class for which it is known that the spectrum is always singular, see \cite{Queff,Bartlett,BG}.  Berlinkov and Solomyak \cite{BerSol} proved that in order to have an absolutely continuous component for the $\Z$-action arising from a constant-length substitution, the substitution matrix must have $\sqrt{\theta_1}$ as an eigenvalue, where $\theta_1$ is its Perron-Frobenius eigenvalue. This is, however, not a necessary condition, as shown by Chan and Grimm \cite{CG}; see also the comments at the end of \cite[Section 3.5]{BG2M}.

In the non-constant length case, new results on singular spectrum for self-similar $\R$-actions were recently obtained by Baake and collaborators \cite{BFGR,BGM,BG2M}, based on estimates of upper Lyapunov exponents for a different version of the cocycle. More specifically, Baake, Frank, Grimm, and Robinson \cite{BFGR} proved singularity of the spectrum for the self-similar suspension in the case of the non-Pisot substitution $0\mapsto 0111, 1\mapsto 0$, which was a first result in this direction and introduced important new ideas. Instead of dimension estimates, they used functional relations obtained from the Riesz product representation of a spectral measure. This result was extended by Baake, Grimm, and Ma\~nibo \cite{BGM} to the substitutions $0\mapsto 01^\ell, 1\mapsto 0$, for $\ell\ge 3$. Their method was further developed and extended to higher dimensional tiling systems by Baake, G\"ahler, Grimm, and Ma\~nibo \cite{BG2M}. However, the singularity of the dynamical spectrum for the corresponding $\Z$-actions remained open. 
 
The  paper is organized as follows. In the next section we recall the background on substitutions, including the definition and basic properties of the spectral cocycle, followed by statement of results. The proofs of main theorems are given in Section 3.  In Section 4 we prove the lemmas and include short proofs of some known results for completeness, and Section 5 is devoted to examples.
{ An Appendix provides the needed facts about the Mahler measure of a polynomial.}

%%%%%%%%%%%%%%%%%%%%%%%%%%%%%%%%%%%%%%%

\section{Background and statement of result}

\subsection{Substitution dynamical systems}
The standard references for the basic facts on substitution dynamical systems are \cite{Queff,Siegel}.
Consider an alphabet of $d\ge 2$ symbols $\Ak=\{0,\ldots,d-1\}$. Let  $\A^+$ be the set of nonempty words with letters in $\A$. 
A {\em substitution}  is a map $\zeta:\,\A\to \A^+$, extended to 
 $\A^+$ and $\A^{\N}$ by
concatenation. The {\em substitution space} is defined as the set of bi-infinite sequences $x\in \A^\Z$ such that any word  in $x$
appears as a subword of $\zeta^n(\a)$ for some $\a\in \A$ and $n\in \N$. The {\em substitution dynamical system}  is the left
shift on $\A^\Z$ restricted to $X_\zeta$, which we denote by $T$.

The {\em substitution matrix} $\Sf=\Sf_\zeta=(\Sf(i,j))$ is the $d\times d$ matrix, such that $\Sf(i,j)$ is the number
of symbols $i$ in $\zeta(j)$. The substitution is {\em primitive} if $\Sf_\zeta^n$ has all entries strictly positive for some $n\in \Nat$.
It is well-known that primitive substitution $\Z$-actions are minimal and uniquely ergodic, see \cite{Queff}.
We assume that the substitution is primitive and {\em non-periodic}, which in the primitive case is equivalent to the space $X_\zeta$ being infinite.
The length of a word $u$ is denoted by $|u|$. The substitution $\zeta$ is said to be of {\em constant length} $q$ if $|\zeta(a)|=q$ for all $a\in \A$, otherwise, it is of {\em non-constant length}.

\subsection{The spectral cocycle}
 A standing assumption will be that $\Sf_\zeta$ is invertible. 
We next recall the definition of the spectral cocycle from \cite{BuSo19}.
Write $\bz = (z_0,\ldots,z_{d-1})$ and $\bz^v = z_{v_0}z_{v_1}\ldots z_{v_k}$ for a word $v = v_0v_1\ldots v_k\in \A^k$.
%Consider  the toral endomorphism $\xi\mapsto \Sf_\zeta^t \,\xi\ (\mbox{mod}\ \Z^m),\ \xi \in \T^m = \R^m/\Z^m$,  induced by the transpose substitution matrix. 
Suppose that
$$\zeta(b) = u_1^{b}\ldots u_{|\zeta(b)|}^{b},\ \ b\in \Ak.$$
Define a matrix-function on $\T^d$ of polynomials in $\bz$-variables by
$$
\Mc_\zeta(\bz) = [\Mc_\zeta(z_0,\ldots,z_{d-1})]_{b,c} = \Bigl( \sum_{j\le |\zeta(b)|,\ u_j^{b} = c} \bz^{u^b_1\ldots u^b_{j-1}} \Bigr)_{(b,c)\in \A^2},\ \ \ \bz\in \T^d,
$$
where $j=1$ corresponds to $\bz^{\es}=1$. 

Whereas the $\bz$-notation is helpful, we will actually need to lift $\Mc_\zeta$ to the universal cover, in other words, write
$\bz = \exp(-2 \pi i \bxi)$, where $\xi = (\xi_0,\ldots,\xi_{d-1})$. Thus we obtain a $\Z^d$-periodic matrix-valued function
 function, which we denote by the same letter.
 
 \begin{defi} \label{def-matrix} Let
  $\Cc_\zeta:  \R^d\to M_d(\C)$  (the space of complex $d\times d$ matrices): 
 \be \label{coc0}
\Cc_\zeta(\xi) = [\Cc_\zeta(\xi_1\ldots,\xi_m)]_{(b,c)} := \Bigl( \sum_{j\le |\zeta(b)|,\ u_j^{b} = c} \exp\bigl(-2\pi i \sum_{k=1}^{j-1} \xi_{u_k^{b}}\bigr)\Bigr)_{(b,c)\in \A^2},\ \ \ \xi\in \R^d.
\ee
%Note that  $\Cc_\zeta$ is $\Z^d$-periodic, so we obtain a continuous matrix-function on the torus, which we denote, by a slight abuse of notation, by the same letter:
%$\Cc_\zeta: \T^d\to M_d(\C)$.
\end{defi}

\noindent {\bf Example.}
Consider the substitution $\zeta: 0\mapsto 000 1,\ 1\mapsto 012,\ 2 \mapsto 1$. Then
$$
\Mc_{\zeta}(\zeta_0,\zeta_1,\zeta_2) = \left( \begin{array}{ccc} 1 + z_0 + z_0^2 & z_0^3 & 0 \\ 1 & z_0 & z_0 z_1 \\ 0 & 1 & 0 \end{array} \right), \ \ \ z_j = e^{-2\pi i \xi_j}.
$$

\medskip

Observe that
\begin{itemize}
\item[(i)] $\Mc_\zeta(0) = \Sf_\zeta^t$;
the entries of the matrix $\Mc_\zeta(\xi)$ are trigonometric polynomials with coefficients 0's and 1's that are $\le$ the corresponding entries of $\Sf^t_\zeta$ in absolute value for every $\bz\in \T^d$;
\item[(ii)] {\em cocycle condition}:
 for any substitutions $\zeta_1,\zeta_2$ on $\Ak$,
\be \label{most}
\Cc_{\zeta_1\circ \zeta_2}(\xi) = \Cc_{\zeta_2}(\Sf^t_{\zeta_1}\xi)\Cc_{\zeta_1}(\xi),
\ee
which is verified by a direct computation.
\end{itemize}

\begin{defi}
Consider the  endomorphism of the torus $\T^d$
\begin{equation}\label{torend}
E_\zeta: \xi \mapsto \Sf_\zeta^t \xi \  (\mathrm{mod} \   \Z^d),
\end{equation} 
 which preserves the Haar measure $m_d$, and define 
 \be \label{cocycle3}
\Cc_\zeta(\xi,n):= \Cc_\zeta\bigl(E_\zeta^{n-1}\xi \bigr)\cdot \ldots \cdot \Cc_\zeta(\xi),
\ee
 a complex matrix cocycle over the endomorphism \eqref{torend}.
 \end{defi}
 
  Note that (\ref{most}) implies
\be \label{coc2}
\Mc_\zeta(\xi,n) =\Cc_{\zeta^n}(\xi),\ \ n\in \N.
\ee

\noindent {\bf Remark.}  The matrix-function $\Cc_{\zeta^n}(\xi)$, with $\xi = \om\vec s$,  was originally introduced in our paper \cite[(4.15)]{BuSo14}, denoted ${\bf M}_n^{\vec s}(\om)$ there.
The matrix $\Mc_\zeta$ appeared in the papers
\cite{BFGR,BGM,BG2M} as a {\em Fourier matrix}, denoted $B(\wtil{k})$; it was further used to define a cocycle on  $\R$, which is the restriction of $\Mc_\zeta(\xi,n)$ to the line in the direction of the Perron-Frobenius eigenvector for $\Sf_\zeta^t$. See Section~\ref{sec-disc} for a more detailed discussion.

\subsection{Formulation of the main result}
 Denote by $\|\cdot\|$  a matrix norm on $M_d(\C)$. In what follows, we can use any norm, since all such norms are equivalent. It will often be convenient for us to use the {\em Frobenius norm}, 
defined by 
$${\|A\|}_{\rm F}^2 = \sum_{i,j} |a_{ij}|^2,$$
which satisfies ${\|AB\|}_F\le {\|A\|}_{\rm F}\cdot {\|B\|}_{\rm F}$. Let $m_d$ be the Haar measure on the torus $\T^d$.

{
\begin{lemma} \label{lem-integrab}
The function $\xi\mapsto \log \|\Mc_\zeta(\xi)\|$ is integrable over $(\T^d,m_d)$,  moreover,
\be \label{eq-integ}
\int_{\T^d} \log \|\Mc_\zeta(\xi)\|\,dm_d(\xi) \ge 0.
\ee
\end{lemma}

\begin{proof}
Using the Frobenius norm, we obtain that ${\|\Mc_\zeta(\xi)\|}_{\rm F}^2$ is a non-trivial trigonometric polynomial in $d$ variables with integer coefficients, and hence
$$
\int_{\T^d} \log {\|\Mc_\zeta(\xi)\|}^2_{\rm F}\,dm_d(\xi) = \mfr(P),
$$
for some polynomial $P\in \Z[z_0,\ldots,z_{d-1}]$, where $\mfr(P)$ is the logarithmic Mahler measure, see \eqref{eq-Mahler}. Therefore, \eqref{eq-Mahler3} implies
\eqref{eq-integ}.
\end{proof}
}

Suppose that $\Sf_\zeta$ has no eigenvalues that are roots of unity. Then the  endomorphism $(\T^d,E_\zeta,m_d)$ is ergodic. By the
Furstenberg-Kesten Theorem \cite{FK}, the (top) Lyapunov exponent exists and is constant almost everywhere:
\be \label{eq-FK}
\exists\, \chi(\Cc_\zeta) = \lim_{n\to\infty} \frac{1}{n}\log\|\Cc_\zeta(\xi,n)\|\ \ \ \mbox{for a.e.}\ \xi\in \T^d.
\ee
%The Lyapunov exponent does not depend on the matrix norm.  

\begin{theorem} \label{thm1}
Let $\zeta$ be a primitive aperiodic substitution on $\Ak = \{0,\ldots,d-1\}$, with $d\ge 2$, such that  the substitution matrix $\Sf_\zeta$ has characteristic polynomial irreducible over $\Q$. Let $\theta_1$ be the Perron-Frobenius eigenvalue of $\Sf_\zeta$. If
\be \label{cond1}
\chi(\Cc_\zeta) < \frac{\log\theta_1}{2},
\ee
then the substitution $\Z$-action has pure singular spectrum.
\end{theorem}

We can also treat some substitutions $\zeta$ on two symbols for which the characteristic polynomial of $\Sf_\zeta$ is reducible.

\begin{theorem} \label{thm2}
Let $\zeta$ be a primitive aperiodic substitution on $\Ak = \{0,1\}$ of non-constant length for which  $\Sf_\zeta$ has two integer eigenvalues, both greater than one in absolute value. If (\ref{cond1}) holds,
then the substitution $\Z$-action has pure singular spectrum.
\end{theorem}

For comparison, we point out the following

\begin{prop}[{\cite[Cor.\,4.7(i)]{BuSo19}}] \label{prop-gap} Suppose that $\Sf_\zeta$ has no eigenvalues that are roots of unity, so that $\chi(\Mc_\zeta)$ is well-defined. Then
$$
\chi(\Cc_\zeta) \le \frac{\log\theta_1}{2}\,.
$$
\end{prop}

Essentially the same result is also proved in a direct and more elementary way in \cite[Th.\,3.29]{BG2M}, and we sketch their proof in Section 4 for completeness.

\medskip

For a practical verification of the condition (\ref{cond1}) one can use the following standard result. 

\begin{lemma} \label{lem-vspom1}
Suppose that $\Sf_\zeta$ has no eigenvalues that are roots of unity, so that $\chi(\Mc_\zeta)$ is well-defined. Then
\be \label{cond2}
\chi(\Mc_\zeta) = \inf_k \frac{1}{k} \int_{\T^d} \log\|\Cc_{\zeta^k}(\xi)\|\,dm_d(\xi).
\ee
\end{lemma}

{
\begin{proof}
Note that $k\mapsto \log\|\Cc_{\zeta^k}(\xi)\|$ is a subadditive sequence of integrable functions. The endomorphism $E_\zeta:\xi \mapsto \Sf_\zeta^t \xi$ (mod $\Z^d$) is ergodic, hence by
Kingman's subadditive ergodic theorem \cite{Kingman}, see also \cite{Steele},
the sequence $\frac{1}{n} \log\|\Cc_{\zeta^n}(\xi)\|$ converges a.e.\ to $\inf_k \frac{1}{k} \int_{\T^d} \log\|\Cc_{\zeta^k}(\xi)\|\,dm_d(\xi)$, and the claim follows.
\end{proof}

\begin{corollary} Suppose that $E_\zeta$ is ergodic. Then $\chi(\Mc_\zeta)\ge 0$.
\end{corollary}

\noindent This is immediate from Lemma~\ref{lem-vspom1} and \eqref{eq-integ}. 

}

\subsection{Discussion} \label{sec-disc} Condition (\ref{cond1}) appeared in \cite[Cor. 4.7(ii),(iii)]{BuSo19}, which asserts that it implies pure singular spectrum for the suspension (special) flow  over the uniquely ergodic substitution
dynamical system $(X_\zeta,T,\mu)$, for almost every piecewise-constant roof function. 
In fact, by \cite[Cor. 4.5(ii)]{BuSo19}, applied to a single substitution, if for a given $\vec s\in \R^d_+$, for Lebesgue-a.e.\ $\om\in \R$,
\be \label{cond3}
\chi^+_\zeta(\om\vec{s}):=\limsup_{n\to \infty} \frac{1}{n} \log\|\Cc_\zeta(\om\vec s, n)\|< \frac{\log\theta_1}{2},
\ee
then the substitution $\R$-action (suspension flow over $(X_\zeta,T,\mu)$), corresponding to the ``roof function'' determined by $\vec s$, has pure singular spectrum.

Related results for the diffraction spectrum of self-similar substitution $\R$-action were recently obtained by Baake et al.\ \cite{BFGR,BGM,BG2M}. More precisely, denote by $\vec{e}_\zeta$ the (normalized) Perron-Frobenius eigenvector of the transpose substitution matrix:
$$
\Sf_\zeta^t \vec{e}_\zeta = \theta_1\vec{e}_\zeta.
$$ 
The self-similar substitution $\R$-action is the suspension flow over $(X_\zeta,T)$ with the ``roof'' given by $\vec{s}= \vec{e}_\zeta$. Alternatively, it can be viewed as a substitution tiling dynamical system on the line, see \cite{BR,SolTil}, where the tiles are line segments of lengths given by the components of $\vec{e}_\zeta$. The {\em diffraction spectrum} has been studied for such tilings, or equivalently, for the Delone sets associated with them. It is known since the work of Dworkin \cite{Dworkin} that the diffraction spectrum is essentially a ``part'' of the dynamical spectrum of the corresponding dynamical system.

In \cite[Th.\ 3.28]{BG2M} it is proved (in different, but equivalent, terms) that if there exists $\eps>0$ such that for Lebesgue-a.e.\ $\om\in \R$,
\be \label{cond4}
\chi^+_\zeta(\om\vec{e}_\zeta) < \frac{\log\theta_1}{2}-\eps,
\ee
then the diffraction spectrum of the self-similar $\R$-action is purely singular. 
The first such non-Pisot, non-constant length example  $0\mapsto 0111,\ 1\mapsto 0$ was analyzed by Baake, Frank, Grimm, and Robinson \cite{BFGR}, and then in \cite{BGM} this was extended to the family of substitutions $0\mapsto 01^\ell,\ 1\mapsto 0$, with $\ell\ge 3$. It is shown there that 
(\ref{cond4}) follows from (\ref{cond2}), and the latter was verified for these examples by a combination of rigorous estimates and numerical computations. 
This approach was further developed and extended to higher-dimensional substitution tiling systems by Baake, G\"ahler, Grimm, and Ma\~nibo \cite{BG2M}.

The main novelty of our results is that they give, for the first time, a proof of singularity of the dynamical spectrum for the $\Z$-actions associated with a class of non-Pisot non-constant length substitutions.

%%%%%%%%%%%%%%%%%%%%%%%%%%%%%%%%%%%%%%%%%%%

\section{Proofs of the main results}

\begin{lemma} \label{lem-singul}
Let $\vec 1 = (1,\ldots,1)^t$. If for Lebesgue-a.e.\ $\om\in \R$,
\be \label{cond5}
\chi^+_\zeta(\om\vec 1)=\limsup_{n\to \infty} \frac{1}{n} \log\|\Cc_\zeta(\om\ve 1, n)\|< \frac{\log\theta_1}{2},
\ee
then the substitution $\Z$-action $(X_\zeta,T,\mu)$ has pure singular spectrum.
\end{lemma}

This is almost immediate from \cite[Cor. 4.5(iii)]{BuSo19} in the special case of a single substitution, as we explain in the next section. We also provide there an alternative proof, essentially equivalent, but simpler ``from scratch'' and more direct, since it relies on \cite{BuSo14} rather than \cite{BuSo19}.

\begin{proof}[Proof of Theorem~\ref{thm1}.] 
The problem in applying Lemma~\ref{lem-singul} is that
the diagonal is a set of Haar measure zero on the torus, and we cannot use (\ref{cond1}) directly, as the Lyapunov exponent exists, a priori, only $m_d$-a.e. We can, however, use a powerful equidistribution theorem of B. Host \cite{Host}  to get around this.

Observe that for any $\vec s\in \R^d_+$  and $k\ge 1$,
\begin{eqnarray}
\chi^+_\zeta(\om\vec{s}) & = & \limsup_{n\to \infty} \frac{1}{nk} \log\|\Cc_{\zeta^k}(\om\vec s, n)\| \nonumber \\
& \le & \limsup_{n\to \infty} \frac{1}{nk}  \sum_{j=0}^{n-1} \log\|\Cc_{\zeta^k}(E_\zeta^{kj} (\om\vec s))\| \nonumber \\
& \le & \lim_{\eps\to 0} \limsup_{n\to \infty} \frac{1}{nk}  \sum_{j=0}^{n-1} \log\bigl(\eps+\|\Cc_{\zeta^k}(E_\zeta^{kj} (\om\vec s))\|\bigr). \label{Lyap}
\end{eqnarray}

We need a theorem of B. Host, namely, a combination of Theorem 2 and Proposition 1 from \cite{Host}, which we restate below in the form convenient for us.

\begin{theorem}[{Host}] \label{th-host}
Let $A$ be a $(d\times d)$-matrix with integer entries, such that for every $r>0$, the characteristic polynomial of $A^r$ is irreducible over $\Q$.
Let $B$ be a $(d\times d)$-matrix with integer entries, such that all eigenvalues of $B$ have modulus $>1$, and $q = |\det(B)|$ is relatively prime with $\det(A)$.
If the probability measure $\mu$ on $\T^d$ is invariant, ergodic and of positive entropy for the endomorphism corresponding to $B$, then the sequence $\{A^n \bt\}_{n\ge 0}$ is equidistributed on $\T^d$ for $\mu$-a.e.\ $\bt\in \T^d$.
\end{theorem}

{ In Theorem~\ref{thm1} we apply Host's Theorem with $A = (\Sf^t_\zeta)^k$, which defines the endomorphism $E^k_\zeta$. We assumed that the characteristic polynomial of $\Sf_\zeta$ is irreducible, and in our case this implies irreducibility for all powers; see the elementary Lemma~\ref{lem-irred} in the next section. Further, $\mu$ is the Lebesgue measure on the diagonal $(x,\ldots,x)$; it is invariant and ergodic, with positive entropy, for $B=pI$ ($I$ is the identity matrix), with $p\ge 2$, which is just the times-$p$ map on the diagonal.  It is trivial to satisfy the condition $GCD(\det A,\det B)=1$ since $p\ge 2$ can be taken arbitrary. The conclusion is that for Lebesgue-a.e.\ $\om$, the
 sequence $\{E_\zeta^{kj}(\om \one)  \}_{j\ge 1}$ is uniformly distributed on the torus $\T^d$. Since $\xi\mapsto \log(\eps + \|\Cc_{\zeta^k}(\xi)\|)$ is a continuous function on $\T^d$ for $\eps>0$, this implies that for Lebesgue-a.e.\ $\om\in \R$, for any fixed $k\in \N$,
\be \label{equidis}
 \lim_{n\to \infty} { \frac{1}{n}} \sum_{j=0}^{n-1} \log\bigl(\eps+\|\Cc_{\zeta^k}(E_\zeta^{kj} (\om \one))\|\bigr) = \int_{\T^d} \log\bigl(\eps + \|\Cc_{\zeta^k}(\xi)\|\bigr)\,dm_d(\xi).
 \ee}
In view of (\ref{Lyap}),   we then obtain for Lebesgue-a.e.\ $\om$ and any fixed $k\ge 1$:
$$
\chi^+_\zeta(\om\vec{1}) \le \lim_{\eps\to 0} \frac{1}{k} \int_{\T^d} \log\bigl(\eps + \|\Cc_{\zeta^k}(\xi)\|\bigr)\,dm_d(\xi)=\frac{1}{k} \int_{\T^d} \log\|\Cc_{\zeta^k}(\xi)\|\,dm_d(\xi).
$$
The passage to the limit $\eps\to 0$ under the integral sign is justified by splitting the integral into the sum: (A) over the set where $\|\Mc_{\zeta^k}(\xi)\|\ge \half$ and (B) over the set where $\|\Mc_{\zeta^k}(\xi)\|\in (0,\half)$. In (A) the functions are uniformly bounded, and in (B) we can use the Dominated Convergence Theorem,  since $\xi \mapsto \log\|\Cc_{\zeta^k}(\xi)\|$ is integrable 
(cf.\ Lemma~\ref{lem-integrab}).
The proof is concluded by an application of Lemmas~\ref{lem-vspom1} and \ref{lem-singul}.
\end{proof}

\begin{proof}[Proof of Theorem~\ref{thm2}] The proof is exactly the same as that of Theorem~\ref{thm1}, except that we use an equidistibution result of D. Meiri, which we quote in a special case relevant for us.

\begin{theorem}[{D. Meiri \cite[Thm. 6.4]{Meiri}}]
Let $ B= \scriptscriptstyle{\left(\begin{array}{cc} p & 0 \\ 0 & p\end{array}\right)}$, for $p\ge 2$ an integer. Suppose that $\mu$ is a Borel $B$-invariant measure on $\T^2$, ergodic and with positive entropy. Let $A$ be an integer $2\times 2$ matrix. Assume that $A$ has two eigenvalues $\lam_1,\lam_2\ne 0$ such that $\det (A)$ and $p$ are relatively prime, and $\lam_1,\lam_2, \lam_1/\lam_2$ are not roots of unity. Furthermore, assume that there exists an entropy-decreasing direction $(a,b)$ for which $\scriptscriptstyle{\left(\begin{array}{c} b \\ -a \end{array}\right)}$ is not an eigenvector of $A$. Then for $\mu$-a.e.\ $(x,y)\in \T^2$, the sequence $\Bigl\{A^n \scriptscriptstyle{\left(\begin{array}{c} x \\ y \end{array}\right)}\Bigr\}$ is uniformly distributed on $\T^2$.
\end{theorem}

In our application, $\mu$ is again the Lebesgue measure on the diagonal $\{(x,x)\}$; it is $B$-invariant and ergodic of positive entropy. 
The projection of $\mu$ onto a straight line in any direction is invariant under the times-$p$ map and has entropy $\log p$, unless the line is orthogonal to the diagonal. Thus
there is only one entropy-decreasing direction $(1,-1)$, and the condition is for $(1,1)^t$ not  to be an eigenvector for $A$. In our application, $A={\sf S}_\zeta^t$, the transpose substitution matrix, so the assumption that $\zeta$ is a non-constant length substitution is equivalent to $(1,1)^t$ not being an eigenvector.  It is again easy to satisfy the condition $GCD(p,\det(A))=1$ since $p\ge 2$ can be taken arbitrary. 
We assumed that both eigenvalues of $\Sf_\zeta$ are integers, since otherwise the case is already covered by Theorem~\ref{thm1}. Having the second eigenvalue greater than one in modulus guarantees that all conditions on the eigenvalues are satisfied.
The conclusion is that for Lebesgue-a.e.\ $x$, 
the sequence $\Bigl\{A^n \scriptscriptstyle{\left(\begin{array}{c} x \\ x \end{array}\right)}\Bigr\}$ mod $\Z^2$ is uniformly distributed on $\T^2$, and the proof is finished as in Theorem~\ref{thm1}, starting from (\ref{equidis}).
\end{proof}

%%%%%%%%%%%%%%

\section{ Proofs of auxiliary results}

\subsection{Alternative proof of Proposition~\ref{prop-gap}}
This proof  follows \cite[Th.\,3.29]{BG2M} and is included for completeness.
We obtain from Lemma~\ref{lem-vspom1} and Jensen's inequality for $k\in \N$:
\be
\chi(\Mc_\zeta) \le \frac{1}{k}\int_{\T^d} \log\|\Mc_{\zeta^k}(\xi)\|\,dm_d(\xi) \le \frac{1}{2k} \log\int_{\T^d} \|\Mc_{\zeta^k}(\xi)\|^2\,dm_d(\xi). 
\ee
Recall that the entries of $(\Mc_{\zeta^k}(\xi))_{b,c}$ are trigonometric polynomials in $d$ variables, and we can use
 the Frobenius norm of a matrix, that is, ${\|(A_{ij})_{i,j}\|}_{\rm F} = (\sum_{i,j} |A_{ij}|^2)^{1/2}$.  By Parseval's formula,
$$
\int_{\T^d} |(\Cc_{\zeta^k}(\xi))_{b,c}|^2\,d\xi = \sum_{\bn \in \Z^d} |(\what{\Cc_{\zeta^k}})_{b,c}(\bn)|^2 = \sum_{\bn \in \Z^d} |(\what{\Cc_{\zeta^k}})_{b,c}(\bn)|.
$$
In the last equality we used the fact that the trigonometric polynomials $(\Mc_{\zeta^k}(\xi))_{b,c}$ have only coefficients $0$ and $1$. Thus we obtain
\be \label{ugu}
\chi(\Cc_\zeta) \le \frac{1}{2k} \log\sum_{b,c} \sum_{\bn \in \Z^d}  |(\what{\Cc_{\zeta^k}})_{b,c}(\bn)| = \frac{1}{2k} \log\sum_{b,c} |\Sf_{\zeta^k}^t(b,c)| = \frac{1}{2k} \log\sum_{b,c} |\Sf_{\zeta^k}(c,b)|,
\ee
where we used that the number of non-zero coefficients of $(\Cc_{\zeta^k}(\xi))_{b,c}$ is exactly the number of letters $c$ in $\zeta^k(b)$, which is the $(c,b)$-entry of the substitution matrix $\Sf_{\zeta^k}$.
It remains to note that by Perron-Frobenius, $|\Sf_{\zeta^k}(c,b)|\sim \theta_1^k$ as $k\to \infty$, and the desired inequality $\chi(\Cc_\zeta) \le \half\log\theta_1$ follows.
\qed

\subsection{Singular spectrum via dimension estimates}

\begin{proof}[First proof of Lemma~\ref{lem-singul}]
It is well-known that the spectral measure of a $\Z$-action on the circle is closely related to the spectral measure on the line of the suspension corresponding to the constant roof function, i.e., when we take $\vec s = \vec 1 = (1,\ldots,1)^t$. See \cite[Prop.\,1.1]{DL} for a general treatment and \cite[Lem.\,5.6]{BerSol} for our specific case. (Incidentally, the relevant result about Fourier transforms  appeared in \cite[Theorem F]{Goldberg}, of which the authors of \cite{BerSol} were unaware at the time of its writing; we are grateful to Nir Lev for the reference.) In particular, singularity of the spectrum for the $\Z$-action is equivalent to the singularity of the spectrum for the corresponding suspension $\R$-action. Now the desired claim follows from
\cite[Cor. 4.5(iii)]{BuSo19} in the special case of a single substitution.
\end{proof}

\begin{proof}[Second proof of Lemma~\ref{lem-singul}] It is known that the spectral type of a primitive aperiodic substitution $\Z$-action $(X_\zeta,T,\mu)$ is
$$
\sum_{k\ge 0,\ a\in \Ak} 2^{-k} {\sig}_{\One_{\zeta^k[a]}},
$$
see \cite[Section 2.4]{BuSo14}, thus it is enough to show that the spectral measure ${\sig}_{\One_{\zeta^k[a]}}$ is singular. Recall that the lower dimension of a measure $\nu$ at $\om$ is defined by
$$
\und{d}(\nu,\om) = \liminf_{r\to 0} \frac{\log\nu(B(\om,r))}{\log r}.
$$
Let $\sig_a = \sig_{\One_a}$. Using a slightly different notation, \cite[Prop.\,7.2]{BuSo14} asserts that for any $\om\in [0,1]$,
\be \label{claim1}
\und{d}(\sig_a,\om) \ge 2 - \frac{2\max\{\chi^+_\zeta(\om\vec 1),0\}}{\log \theta_1}\,.
\ee
(The inequality in \cite[Prop.\,7.2]{BuSo14} is $\und{d}(\sig_a) \ge 2 - 2\chi^+_\zeta(\om\vec 1)/\log \theta_1$, but actually (\ref{claim1}) is proved there.) The same argument, using the discussion of
\cite[Section 2.4]{BuSo14}, gives, for any $\om\in [0,1]$ and $k\ge 0$,
\be \label{claim2}
\und{d}(\sig_{\One_{\zeta^k[a]}},\om) \ge 2 - \frac{2\max\bigl\{\chi^+_\zeta\bigl( E^k_\zeta (\om\vec 1)\bigr),0\bigr\}}{\log \theta_1}\,.
\ee
We claim that
\be \label{claim3}
\chi^+_\zeta\bigl( E^k_\zeta (\om\vec 1)\bigr) = \chi^+_\zeta(\om\vec 1)\ \ \mbox{for all}\ \om\in [0,1]\setminus \Ek,\ \mbox{with}\ \Ek\ \mbox{countable}.
\ee
Indeed,
\begin{eqnarray*}
\chi^+_\zeta\bigl( E^k_\zeta (\om\vec 1)\bigr) & = & \limsup_{n\to \infty} \frac{1}{n} \log\bigl\|\Cc_\zeta\bigl(E_\zeta^{n+k-1}(\om\vec 1) \bigr)\cdot \ldots \cdot \Cc_\zeta(E_\zeta^k(\om \vec 1))\bigr\| \\[1.2ex]
& = & \chi^+_\zeta(\om\vec 1),
\end{eqnarray*}
provided
$$
\det\bigl[\Mc_\zeta\bigl(E_\zeta^{k-1}(\om\vec 1) \bigr)\cdot \ldots \cdot \Cc_\zeta(\om \vec 1)\bigr]\ne 0.
$$
However, $\om\mapsto \det\Mc_\zeta(E_\zeta^\ell( \om\vec 1))$ is a nontrivial trigonometric polynomial of a single variable for any $\ell\in \N$, since $\det \Mc_\zeta(0) = \det \Sf_\zeta^t \ne 0$, and so it has countable many zeros. The claim (\ref{claim3}) follows.

Now we can conclude, repeating the argument of \cite[Cor.\,4.5(ii)]{BuSo19}, as follows. Let $f=\One_{\zeta^k[a]}$ for some $a\in \Ak$ and $k\ge 0$. By the assumption (\ref{cond5}), (\ref{claim2}) and (\ref{claim3}),
\be \label{claim4}
\und{d}(\sig_f,\om)  > 1\ \ \mbox{for Lebesgue-a.e.}\ \om\in [0,1].
\ee
Let $\sig_f = h\cdot m_1 + (\sig_f)_{ \rm sing}$, with $h\in L^1[0,1]$, be the Lebesgue decomposition. By (\ref{claim4}), 
$$
\lim_{r\to 0} \frac{\sig_f(B(\om,r))}{2r} = 0\ \ \mbox{for Lebesgue-a.e.}\ \om\in [0,1],
$$
hence $h=0$ a.e., and $\sig_f$ is purely singular, as desired.
\end{proof}

\medskip

%%%%%%%%%%%%%%%%%%%%%%%%%%%%%%%%%%%%%

\subsection{Total irreducibility} We need the following elementary lemma.

\begin{lemma} \label{lem-irred}
Let $A$ be an integer matrix with an irreducible (over $\Q$) characteristic polynomial $p(x)$. Let $n\in \N$ and  suppose that one of the eigenvalues satisfies
\be \label{condit1}
\theta_1^n \ne \theta_j^n,\ \ \mbox{for all}\ j>1.
\ee
Then $A^n$ has an irreducible characteristic polynomial.
\end{lemma}

Conditon (\ref{condit1}) obviously holds when $A$ is a primitive matrix, since it has a dominant eigenvalue.

\begin{proof}
Let
$$
p(x) = \prod_{j=1}^d (x-\theta_j),\ \ \ F_n(x) =\prod_{j=1}^d (x-\theta_j^n),
$$
so that $F_n$ is the characteristic polynomial of $A^n$. Suppose that $F_n$ is reducible over $\Q$, then we have a factorization into a non-trivial product of polynomials in $\Z[x]$:
$$
F_n(x) = q(x) r(x),\ \ q(x) = \prod_{j=1}^k (x-\theta_j^n),\ \ r(x) = \prod_{j=k+1}^d (x-\theta_j^n),
$$
for an appropriate enumeration of the roots. Then $r(x^n) \in \Z[x]$ has $\theta_d$ as a root, hence $\theta_1$ is its root as well, being a conjugate of $\theta_d$. However,
$$
r(\theta_1^n) =  \prod_{j=k+1}^d (\theta^n_1-\theta_j^n) \ne 0,
$$
by the assumption (\ref{condit1}),
and the contradiction concludes the proof.
\end{proof}

%%%%%%%%%%%%%%%%%%%%%%%%%%%

\section{Examples} 

\subsection{Non-Pisot substitutions with singular spectrum}

\begin{example} \label{ex1} {\em
Consider the substitution $\zeta_m: 0\mapsto 0^m 1,\ 1\mapsto 012,\ 2 \mapsto 1$, for $m$ sufficiently large (say, $m\ge 3$). The characteristic polynomial of the matrix $\Sf_\zeta$ is $$p(x) = -\bigl(x^3 - (m+1)x^2 + (m-2)x + m\bigr).$$ One can check that $p$ has three real roots: $\theta_1 \in (m,m+1)$ (close to $m$ for large $m$), $\theta_2 \in (1,2)$, and 
$\theta_3 \in (-1,0)$. This is therefore a non-Pisot substitution. It is easy to see that $p$ is irreducible.

The matrix-valued function defining the spectral cocycle is
$$
\Mc_{\zeta_m}(z_0,z_1,z_2) = \left( \begin{array}{ccc} 1 + z_0 + \cdots + z_0^{m-1} & z_0^m & 0 \\ 1 & z_0 & z_0 z_1 \\ 0 & 1 & 0 \end{array} \right), \ \ \ z_j = e^{-2\pi i \xi_j}.
$$
In order to verify (\ref{cond2}), we can use the method of Baake, Grimm, and Ma\~nibo, see \cite[Section 5.1 and Appendix]{BGM}. In order to get sharp results, numerical estimates are needed, but in order to obtain a crude bound, we can use $k=1$ in the Frobenius norm of the matrix, which gives
$$
{\|\Mc_{\zeta_m}(\xi)\|}_F^2 = 5 + {\Bigl|\frac{z_0^m-1}{z_0-1}\Bigr|}^2.
$$
Since there is no dependence on $z_1$ and $z_2$, integration reduces to one-dimensional.
Using a crude bound $|z_0^m-1|\le 2$ and writing $z:=z_0=e^{-2\pi i t}$ to simplify notation, we obtain
\be \label{mahler}
\Ik:=\int_{\T^3} \log {\|\Mc_{\zeta_m}(\xi)\|}_F^2\,dm_3(\xi) \le \int_0^1 \log(5|z-1|^2 + 4) \,dt - \int_0^1 \log(|z-1|^2) \,dt.
\ee
Since $5|z-1|^2 + 4=5(2 -z - z^{-1})+4 = |5z^2-14z + 5|$, the right-hand side of (\ref{mahler}) equals
$$
\mfr(5z^2-14z + 5) - \mfr((z-1)^2) = \log(7 + 2\sqrt{6}),
$$
{ using \eqref{eq-Mahler2},
where $\mfr(p)$ is the logarithmic Mahler measure of a polynomial $p$, see \eqref{eq-Mahler}.}
Thus,
$$
\chi(\Mc_{\zeta_m}) \le \half \int_{\T^3} \log {\|\Mc_{\zeta_m}(\xi)\|}_F^2\,dm_3(\xi)\le \half \cdot \log(7 + 2\sqrt{6}).
$$
It follows that $\chi(\Mc_{\zeta_m}) \le \half \log  m < \half \log \theta_1$ for $m\ge 12 > 7+2\sqrt{6}$, and so for such substitutions $\zeta_m$ we obtain pure singular spectrum by 
Theorem~\ref{thm1}.
}
\end{example}

The same method can be extended to many other families of substitutions. We start with a general lemma. Let $\Ak = \{0,\ldots,d-1\}$ be again an alphabet of arbitrary size $d\ge 2$. For $\alpha\ne \beta$ consider the following substitutions on $\Ak$:
\be \label{special}
\zeta_{\alpha\beta}: \left\{ \begin{array}{ll}  j\mapsto j & (j\ne \alpha) \\ \alpha \mapsto \alpha\beta\end{array} \right.;\ \ \ \ \ \ \ \  
\wt\zeta_{\alpha\beta}: \left\{ \begin{array}{ll}  j\mapsto j & (j\ne \alpha) \\ \alpha \mapsto \beta\alpha\end{array} \right.
\ee

\begin{lemma} \label{lem-ex1}
{\rm (i)} Let $\zeta^{(n)} = \sig_1\circ \zeta_{\alpha\beta}^n\circ\sig_2$ or $\zeta^{(n)} = \sig_1\circ \wt\zeta_{\alpha\beta}^n\circ\sig_2$ where $\sig_1$ and $\sig_2$ are arbitrary substitutions on $\Ak$,
such that the substitution matrices $\Sf_{\sig_j}$ are non-singular, $j=1,2$.
Then
\be \label{estim}
\int_{\T^d} \log\bigl\|\Mc_{\zeta^{(n)}}(\xi)\bigr\|_F\,dm_d(\xi) \le \half\log\bigl(d+2 + 2\sqrt{d+1}\bigr) +  \sum_{j=1}^2 \log{\|\Sf_{\sig_j}\|}_F .
\ee

 {\rm (ii)} Suppose, in addition, that least one of $\sig_1,\sig_2$ has a substitution matrix with all entries strictly positive, and  the characteristic polynomial of  $\Sf_{\zeta^{(n)}}$ is irreducible.
% Then, for $n\in \N$ sufficiently large, holds 
%\be \label{conduk1}
%\int_{\T^d} \log\bigl\|\Mc_{\zeta^{(n)}}(\xi)\bigr\|\,dm_d(\xi) < \frac{\log\theta(\zeta^{(n)})}{2}\,,
%\ee
Then the spectrum of the $\Z$-action $(X_{\zeta^{(n)}},\mu,T)$ is pure singular, provided that
\be \label{conduk2}
n+1 > \bigl(d+2 + 2\sqrt{d+1}\bigr) \cdot {\|\Sf_{\sig_1}\|}_F^2 \cdot {\|\Sf_{\sig_2}\|}_F^2.
\ee
\end{lemma}

\begin{proof} (i)
We  will only consider the case of $\zeta^{(n)}  = \sig_1\circ \zeta_{\alpha\beta}^n\circ\sig_2$; the second case is proved in exactly the same way. We have by the definition of the spectral cocycle and invariance of the measure $m_d$ under toral endomorphisms:
\begin{eqnarray}
\int_{\T^d} \log\bigl\|\Mc_{\zeta^{(n)}}(\xi)\bigr\|_F\,dm_d(\xi) & = & \int_{\T^d} \log\bigl\|\Mc_{\sig_2}(\Sf^t_{\zeta_{\alpha\beta}^n}\Sf_{\sig_1}^t\xi)\Mc_{\zeta_{\alpha\beta}^n}(\Sf_{\sig_1}^t\xi)\Mc_{\sig_1}(\xi)\bigr\|_F\,dm_d(\xi) \nonumber  \\
& \le & \int_{\T^d} \log\bigl\|\Mc_{\zeta_{\alpha\beta}^n}(\xi)\bigr\|_F\,dm_d(\xi) + \sum_{j=1}^2\int_{\T^d} \log\bigl\|\Mc_{\sig_j}(\xi)\bigr\|_F\,dm_d(\xi) \nonumber \\
& \le & \int_{\T^d} \log\bigl\|\Mc_{\zeta_{\alpha\beta}^n}(\xi)\bigr\|_F\,dm_d(\xi) + \sum_{j=1}^2 \log\bigl\|\Sf_{\sig_j}\bigr\|_F\,. \label{lia1}
\end{eqnarray}
The last inequality holds since all entries of the matrix-function $\Mc_\xi(\xi)$ are dominated in absolute value by the corresponding entries of $\Sf_\zeta^t$, for any substitution $\zeta$.
Now, the entries of the matrix-function $\Mc_{\zeta_{\alpha\beta}^n}(\xi)$ are
$$
\bigl[\Mc_{\zeta_{\alpha\beta}^n}(\xi)\bigr]_{jk} = \left\{ \begin{array}{cl} 1, & \mbox{if}\ j=k; \\ z_\alpha(1 + z_\beta + \cdots + z_\beta^{n-1}), & \mbox{if}\ j=\alpha,\ k=\beta; \\
0, & \mbox{else,} \end{array} \right.
$$
where $z_j = e^{-2\pi i \xi_j}$ for $j\in \Ak$.
We can therefore estimate, as in Example~\ref{ex1}, using that $|z_\beta^n-1|\le 2$ and then writing $z_\beta = e^{-2\pi i t}$,
\begin{eqnarray}
\int_{\T^d} \log {\|\Mc_{\zeta_{\alpha\beta}^n}(\xi)\|}_F^2\,dm_d(\xi) & = & \int_0^1 \log\left(d + {\Bigl|\frac{z_\beta^n-1}{z_\beta-1}\Bigr|}^2\right)\,dt \nonumber\\
& \le & \int_0^1 \log(d|e^{2\pi i t}-1|^2 + 4) \,dt - \int_0^1 \log(|e^{2\pi i t}-1|^2) \,dt \nonumber \\
& = & \mfr\bigl[dz^2 - (2d+4)z + d\bigr] -\mfr\bigl[(z-1)^2\bigr] = \log\bigl[(d+2)+2\sqrt{d+1}\bigr], \nonumber
\end{eqnarray}
and (\ref{estim}) follows, in view of (\ref{lia1}).

\medskip

(ii) The desired claim will follow from Theorem~\ref{thm1} and Lemma~\ref{lem-vspom1}, once we show that
$\theta_1(\Sf_{\zeta^{(n)}})\ge n+1$. To prove the latter, note that the substitution matrix $\Sf = \Sf_{\zeta^{(n)}}$ has  every entry greater or equal to $n+1$
in at least one of its rows or columns, depending on whether $\Sf_{\sig_1}$ or $\Sf_{\sig_2}$ is strictly positive. Suppose that there is such a column, otherwise, pass to the transpose matrix. We then clearly have $\Sf\,\vec{1}\ge (n+1)\vec{1}$, where $\vec 1 = (1,\ldots,1)^t$ and the inequality is understood component-wise. Thus $\Sf^k \vec 1 \ge (n+1)^k \vec 1$ for $k\in \N$, which implies the claim about the Perron-Frobenius eigenvalue and
completes the proof of the lemma.
\end{proof}

\subsection{A family of self-similar interval exchanges}
Interval exchange transformations (IET's) form an important and much studied class of dynamical systems. An interval exchange $f(\lam,\pi)$ on $d$ intervals is determined by a permutation $\pi \in \FrS_d$ and a positive vector $\lam = (\lam_1,\ldots,\lam_d) \in \R^d_+$. It is a piecewise isometry of an interval $I=\bigcup_{j=1}^d I_j$, where $I_j = \bigl[\sum_{k<j}\lam_k, \sum_{k\le j}\lam_k\bigr)$, for
$j=1,\ldots,d$, in which the intervals $I_j$ are translated and exchanged according to the permutation $\pi$. 
To be precise,
$$
f(\lam,\pi):\ x\mapsto x + \sum_{\pi(j)<\pi(i)} \lam_j - \sum_{j < i} \lam_j,\ \ x\in I_i,
$$
which means that {\em after the exchange} the interval $I_{j}$ is in the $\pi(j)$-th place. Here we are using the convention of Veech \cite{Veech82}; some authors use different notation.
Assume that the permutation is irreducible, i.e., $\pi\{1,\ldots,k\}\ne \{1,\ldots,k\}$ for $k<d$. 
We do not review the basic facts on IET's, but refer the reader to the vast literature; e.g., \cite{Viana,Yoccoz,Zorich}. Whereas a lot is known about ``typical'' or ``generic'' IET's, specific examples are often harder to analyze. An IET $f=f(\lam,\pi)$ on $[0,1)$ is called {\em self-similar} if there is $\alpha \in (0,1)$ such that the first return map of $f$ to $[0,\alpha)$, after rescaling, is equal to $f$. Such IET's were studied, e.g., in \cite{BSU,CG97,CGM}. They are also closely related to {\em pseudo-Anosov diffeomorphisms}, see, e.g.,  \cite{Veech82, Lanneau}. It is well-known that a self-similar IET is almost-topologically conjugate to a primitive substitution. We recall the construction briefly, following \cite{CG97,CGM}.

Suppose that $f$ is a self-similar IET on $[0,1)$. Denote $I_j^{(1)} = \alpha I_j \subset [0,\alpha)$ for $j\in \Ak = \{1,\ldots,d\}$. Let $R$ be the renormalization matrix given by
$$
R_{i,j}  = \#\left\{0 \le k \le r_j-1:\ f^k(I_j^{(1)}) \subset I_i\right\},
$$
where $r_j$ is the first return time of $I_j^{(1)}$ to $[0,\alpha)$. The substitution $\zeta_f:\Ak\to \Ak$ is defined by
\be \label{ietsub}
\zeta_f(j) = u_0\ldots u_{r_j-1},\ \ \ \mbox{where}\ \ \ u_k = i \iff f^k(I_j^{(1)}) \subset I_i,
\ee
so that $R=\Sf_{\zeta_f}$ is the substitution matrix.
One can show that $\zeta_f$ is well-defined and minimal, and the substitution dynamical system $(X_{\zeta_f},T)$ is conjugate to $f$ restricted to $[0,1)$ minus the union of orbits of discontinuities of $f$, see \cite{CG97,CGM}.

Self-similar IET's correspond to cycles in the {\em Rauzy diagram}, see \cite{Rauzy,Veech82,Lanneau}.  We recall the definition of {\em Rauzy induction} (Rauzy algorithm) in brief.
Let $\pi$ be an irreducible permutation, and suppose that $(\lam,\pi)$ is such that $\lam_d \ne \lam_{\pi^{-1}(d)}$. Then the first return map of $f(\lam,\pi)$ to the interval
$$
\Bigl[0, \sum_{i\le d } \lam_i - \min\{\lam_{\pi^{-1}(d)},\lam_d\} \Bigr)
$$
is an irreducible IET on $d$ intervals as well, see, e.g., \cite{Viana,Yoccoz,Zorich}. If $\lam_d < \lam_{\pi^{-1}(d)}$, we say that this is an operation of type ``$a$''; otherwise, an operation of type ``$b$''. The {\em Rauzy diagram} is a directed labeled graph, whose vertices are permutations of $\Ak=\{1,\ldots,d\}$ and the edges lead to permutations obtained by applying one of the operations. The edges are labeled ``$a$'' or ``$b$'' depending on the type of the operation. One can easily find the substitutions corresponding to each Rauzy move, according to the same rule as in (\ref{ietsub}). Denoting by $\zeta_{\pi,a}$ and $\zeta_{\pi,b}$ the substitution corresponding to the ``$a$'' and ``$b$'' moves from the permutation $\pi$ correspondingly, we obtain:
\be\label{def-substi}
\zeta_{\pi,a}:\ \left\{ \begin{array}{cl} i\mapsto i, & i \le \pi^{-1}d; \\
                                                         (\pi^{-1}d+1) \mapsto (\pi^{-1}d)d, & \\
                                                         i\mapsto i-1, & \pi^{-1}d+1 < i \le d; \end{array} \right.\ \ \ 
\zeta_{\pi,b}:\ \left\{ \begin{array}{cl} i\mapsto i, & i \ne \pi^{-1}d; \\
                                                         \pi^{-1}d\mapsto (\pi^{-1}d)d. \end{array} \right.                                                         
\ee

\begin{example}{\em 
Our family of examples is a modification of the ones considered by Lanneau \cite[Section 3]{Lanneau}; in fact, they coincide for $n=1$. We consider the closed loop $\gam$ in the Rauzy diagram based at $\pi = (4321)$, defined by the Rauzy moves
$$
\gam = b\to (a\to a) \to b \to (a)^n \to b\to (a\to a\to a),
$$
where parentheses indicate ``simple'' loops of the diagram. More explicitly, this is the cycle
$$
(4321) \stackrel{\textstyle b}{\longrightarrow} (2431) \stackrel{\textstyle a}{\longrightarrow} (2413) \stackrel{\textstyle a}{\longrightarrow} (2431) \stackrel{\textstyle b}{\longrightarrow}
 \stackrel{\textstyle{\ \ \ \curvearrowright a^n}}{(3241)} \ 
\stackrel{\textstyle b}{\longrightarrow} (4321) \stackrel{\textstyle a}{\longrightarrow} (4132) \stackrel{\textstyle a}{\longrightarrow} (4213) \stackrel{\textstyle a}{\longrightarrow} (4321) 
$$
See, e.g., \cite[p.581]{Bufetov_jams} for the full diagram.
Using the formulas (\ref{def-substi}) and
 composing the substitutions we obtain
\be\label{def-sub}
\zeta(\gam) = \zeta^{(n)} := \zeta_{14}\circ \sig_3^2\circ \zeta_{24}\circ {\bigl(\wt\zeta_{43}\bigr)}^n \circ \zeta_{34}\circ \sig_2^3: \ \ \ \begin{array}{l} 1 \mapsto 14 \\ 2 \mapsto 14224 \\ 3\mapsto 14(23)^{n+1}24\\
4\mapsto 14(23)^n 24\end{array}
\ee
as follows from the diagram below. We intentionally put the arrows pointing from right to left,  since this is more convenient for computing the composition of substitutions:
$$
\begin{array}{r} b\ \ \  \\14 \mapsfrom 1 \\ 2 \mapsfrom 2 \\ 3 \mapsfrom 3 \\ 4\mapsfrom 4 \\[1.1ex] \zeta_{14} \ \end{array}\ 
\begin{array}{r} a\ \ \  \\ 1 \mapsfrom 1 \\ 2 \mapsfrom 2 \\ 24 \mapsfrom 3 \\ 3\mapsfrom 4 \\[1.1ex] \sig_3\ \ \end{array}\ 
\begin{array}{r} a\ \ \  \\ 1 \mapsfrom 1 \\ 2 \mapsfrom 2 \\ 24 \mapsfrom 3 \\ 3\mapsfrom 4 \\[1.1ex] \sig_3\ \ \end{array}\ 
\begin{array}{r} b\ \ \  \\1 \mapsfrom 1 \\ 24\mapsfrom 2 \\ 3 \mapsfrom 3 \\ 4\mapsfrom 4 \\[1.1ex] \zeta_{24}\ \ \end{array}\ 
\begin{array}{r} a^n\ \  \\ 1 \mapsfrom 1 \\ 2 \mapsfrom 2 \\ 3 \mapsfrom 3 \\ 3^n4\mapsfrom 4 \\[1.1ex] {\bigl(\wt\zeta_{43}\bigr)}^n \end{array}\ 
\begin{array}{r} b\ \ \  \\1 \mapsfrom 1 \\ 2\mapsfrom 2 \\ 34 \mapsfrom 3 \\ 4\mapsfrom 4 \\[1.1ex] \zeta_{34} \ \ \end{array}\ 
\begin{array}{r} a\ \ \  \\ 1 \mapsfrom 1 \\ 14 \mapsfrom 2 \\ 2 \mapsfrom 3 \\ 3\mapsfrom 4 \\[1.1ex] \sig_2\ \  \end{array}\ 
\begin{array}{r} a\ \ \  \\ 1 \mapsfrom 1 \\ 14 \mapsfrom 2 \\ 2 \mapsfrom 3 \\ 3\mapsfrom 4 \\[1.1ex] \sig_2\ \ \end{array}\ 
\begin{array}{r} a\ \ \  \\ 1 \mapsfrom 1 \\ 14 \mapsfrom 2 \\ 2 \mapsfrom 3 \\ 3\mapsfrom 4 \\[1.1ex] \sig_2\ \ \end{array}\ 
$$
Here the notation for the substitutions $\zeta_{\alpha\beta}$ and $\wt\zeta_{\alpha\beta}$ agrees with that of (\ref{special}); the notation for $\sig_2$ and $\sig_3$ is introduced here for convenience. Notice that sometimes the same substitution corresponds to different ``arrows'' in the Rauzy diagram. The substitution matrix is
$$
\Sf_{\zeta^{(n)}} = \left(\begin{array}{cccc} 1 & 1 & 1 & 1 \\ 0 & 2 & (n+2) & (n+1) \\ 0 & 0 & (n+1)& n \\ 1 & 2 & 2 & 2 \end{array} \right),
$$
and its characteristic polynomial is
$$
p_n(x) = x^4 - (n+6) x^3 + (10+n) x^2 - (n+6) x + 1.
$$
Notice that this is a reciprocal polynomial, as characteristic polynomials of all renormalization matrices of the Rauzy induction, which follows from the fact that it preserves a certain symplectic form; see \cite{Veech82,Viana,Yoccoz}. 
{ The roots of $p_n$ can be found by a well-known method; see e.g.\ \cite[Section 6.7]{MaMoYo}.
Namely, $\lam$ is a root of $x^4 + ax^3 + bx^2 + ax + 1$ if and only if $\lam+ \lam^{-1}$ is a root of $y^2 +ay + (b-2)$, so that if  $\lam$ is a root of $p_n(x)$, then
$\lam+ \lam^{-1}$ is a root of $y^2 - (n+6) y - (n+8)$, which implies
$$
\lam+ \lam^{-1} - 2 = \frac{(n+2) \pm \sqrt{(n+2)^2 + 4n}}{2}.
$$
It follows that the polynomial $p_n$ is irreducible, and $\theta_1$ is a Salem number for all $n\ge 1$, 
since $n+2 < \sqrt{(n+2)^2 + 4n} < (n+2) + 2$. We are grateful to the referee who suggested this argument.
}
}
\end{example}

From Lemma~\ref{lem-ex1}(i) we now obtain

\begin{corollary}
The $\Z$-action corresponding to the substitution $\zeta^{(n)}$, defined in (\ref{def-sub}), is weakly mixing for all $n\ge 1$. The spectrum is purely singular, for all
$$
n \ge 922  > (6+2\sqrt{5})\cdot 88.
$$
\end{corollary}

This is, of course, far from being sharp, and shows limitations of our method. It is easy to improve the numbers somewhat, but a substantial improvement would require computer use.

\begin{proof}
The result on weak mixing, that is, absence of a non-trivial discrete spectrum, is well-known; in fact, the same argument works for all ``Salem'' substitutions. %It follows from results of Host \cite{Host-subs} on eigenvalues for substitution systems.
 In fact,
we can use a corollary of Host's characterization of eigenvalues for substitution $\Z$-actions, stated in  \cite[(6.3)]{Host86}, that if the minimal polynomial of the Perron-Frobenius eigenvalue $\theta_1$ has more than half of its roots outside the open unit disk, then there are no eigenvalues of the form $\exp(2\pi i\alpha)$, with $\alpha\not\in \Q$. In our case only one root, $\theta_4$, has absolute value less than one, hence the criterion applies. As for eigenvalues $\exp(2\pi i\alpha)$, with $\alpha\in \Q$,   by \cite[(6.3)]{Host86} again, it is necessary that $\alpha\cdot |(\zeta^{(n)})^m(2)| \in \N$ for $m$ sufficiently large.
The sequence $\{\alpha\cdot |(\zeta^{(n)})^m(2)|\}_{m=0}^\infty$ satisfies the linear recurrence relation with the characteristic polynomial $p_n$, and since $p_n$ has constant terms equal to one, we can use the linear recurrence ``backwards'' to obtain that $\alpha\cdot |(\zeta^{(n)})^0(2)|=\alpha\in \N$, hence the eigenvalue is trivial. This completes the proof of weak mixing.

For the proof of singularity, we apply Lemma~\ref{lem-ex1}(i). Here $d=4$, and we obtain 
\begin{eqnarray*}
\int_{\T^4} \log\bigl\|\Mc_{\zeta^{(n)}}(\xi)\bigr\|_F\,dm_d(\xi) 
& \le &  \half\log\bigl(6 + 2\sqrt{5}\bigr) +  \log{\|\Sf_{\zeta_{14}\sig_3^2\zeta_{24}}\|}_F + \log{\|\Sf_{\zeta_{34} \sig_2^3}\|}_F  \\
& = & \half\log\bigl(6 + 2\sqrt{5}\bigr) + \half(\log 11 + \log 8).
\end{eqnarray*}
It remains to apply Theorem~\ref{thm1} and Lemma~\ref{lem-vspom1}, keeping in mind that $\theta_1 > n$.
\end{proof}

\medskip

{
\appendix \section{On the Mahler measure of a polynomial} Mahler \cite{Mahler} defined the  measure of a polynomial in several variables as follows:
$$
%\be \label{eq-Mahler}
M(P):= \exp \int_{\T^n} \log|P(z_1,\ldots,z_n)| \,d\bt,
$$
%\ee
where $\bt = (t_1,\ldots,t_n)$ and $z_j = e^{2\pi i t_j}$. Now $M(P)$ is called the {\em Mahler measure} and  
\be \label{eq-Mahler}
\mfr(P) = \log M(P)
\ee
is the {\em logarithmic Mahler measure} of the polynomial $P$.
For a polynomial of a single variable $p(z)$, Jensen's formula yields
\be \label{eq-Mahler2}
M(p) = \exp \int_0^1 \log|p(e^{2\pi i t}|\,dt = |a_0|\cdot \prod_{j\ge 1} \max\{|\alpha_j|,1\},
\ee
where $a_0$ is the leading coefficient and $\alpha_j$ are the complex zeros of $p$. Observe that $M(p)\ge |a_0|$, hence writing
$$
P(z_1,\ldots,z_n) = a_0(z_1,\ldots,z_{n-1}) z_n^k + \cdots + a_k(z_1,\ldots,z_{n-1})
$$
and integrating $\log |P(z_1,\ldots,z_n)|$  with respect to $t_n$, we obtain $M(P)\ge M(a_0)$, and then by induction it follows that
\be \label{eq-Mahler3}
\mbox{\em if}\ P(z_1,\ldots,z_n)\ \mbox{\em has integer coefficients, then}\  M(P)\ge 1,
\ee
see e.g., \cite[p.\,117]{Boyd} and \cite{EW} for details and for additional information on the Mahler measure.
}

\bigskip

{\bf Acknowledgement:} We are deeply grateful to Mike Hochman for fruitful discussions of equidistribution results, to Adam Kanigowski for helpful comments, to Nir Lev for the reference to the paper by R. Goldberg, to Misha Sodin, Rotem Yaari, and the anonymous referee for valuable remarks which helped improve the article.
The research of A. Bufetov has received funding from the European Research Council (ERC) under the European Union's Horizon 2020 
research and innovation programme under grant agreement No 647133 (ICHAOS).  A. Bufetov has also been funded by ANR grant  ANR-18-CE40-0035 REPKA.
The research of B. Solomyak was  supported by the Israel Science Foundation (grant 911/19).
B.S. gratefully acknowledges support from the Simons Center for Geometry and Physics, Stony Brook University, where some of the research for this paper was done.

\end{document}